\documentclass[12pt]{amsart}

\usepackage[latin1]{inputenc}					
\usepackage{amsmath,amsfonts,amssymb,amsthm}	

\usepackage{mathtools}		
\usepackage{hyperref}		
\usepackage{setspace}		
\usepackage{cite}			

\onehalfspace
\hypersetup{colorlinks=true,citecolor=blue,linkcolor=blue,urlcolor=magenta}

\numberwithin{equation}{section}			
\newtheorem{theorem}{Theorem}
\newtheorem{proposition}{Proposition}
\newtheorem{lemma}{Lemma}
\newtheorem{corollary}{Corollary}

\newtheorem{definition}{Definition}

\newtheorem*{theorem*}{Theorem}
\newtheorem*{proposition*}{Proposition}
\newtheorem*{lemma*}{Lemma}
\newtheorem*{corollary*}{Corollary}

\newtheorem*{conjecture*}{Conjecture}
\newtheorem*{question*}{Question}
\newtheorem*{remark*}{Remark}
\newtheorem*{definition*}{Definition}

\renewcommand{\leq}{\leqslant}	
\renewcommand{\geq}{\geqslant}	

\newcommand{\N}{\mathbb{N}}		
\newcommand{\Z}{\mathbb{Z}}		
\newcommand{\R}{\mathbb{R}}		
\newcommand{\C}{\mathbb{C}}		
\newcommand{\F}{\mathbb{F}}		
\newcommand{\T}{\mathbb{T}}		
\newcommand{\U}{\mathbb{U}}		
\newcommand{\D}{\mathbb{D}}		

\newcommand{\eps}{\varepsilon}			
\newcommand{\ssum}{\textstyle\sum}		
\newcommand{\cj}{\overline}				
\DeclareMathOperator{\re}{Re}			

\newcommand{\wt}{\widetilde}		
\newcommand{\wh}{\widehat}			
\DeclareMathOperator{\Spec}{Spec}	
\DeclareMathOperator{\Supp}{Supp}	

\newcommand{\E}{\mathbb{E}}			

\newcommand{\ccap}{\textstyle\bigcap}	
\newcommand{\ccup}{\textstyle\bigcup}	

\newcommand{\dual}{\widehat{G}}				

\newcommand{\wf}{\widehat{f}}				
\newcommand{\wg}{\widehat{g}}				
\newcommand{\wmu}{\widehat{\mu}}			
\newcommand{\wone}{\widehat{1}}				

\newcommand{\cB}{\overline{B}}				
\newcommand{\hB}{\widehat{B}}				
\newcommand{\mB}{\mathcal{B}}				
\newcommand{\cmB}{\overline{\mathcal{B}}}	
\newcommand{\wmB}{\widetilde{\mathcal{B}}}	
\newcommand{\hmB}{\widehat{\mathcal{B}}}	
\newcommand{\dmB}{\dot{\mathcal{B}}}		

\newcommand{\mM}{\mathcal{M}}				

\newcommand{\hA}{\widehat{A}}				

\begin{document}

\title{Arithmetic progressions in sets of small doubling}

\author{Kevin Henriot}


\maketitle

\begin{abstract}
  We show that if a finite, large enough 
  subset $A$ of an arbitrary abelian group
  satisfies the small doubling condition
  $|A + A| \leq (\log |A|)^{1 - \eps} |A|$,
  then $A$ must contain a three-term arithmetic progression
  whose terms are not all equal,
  and $A + A$ must contain an 
  arithmetic progression or a coset of a subgroup, 
  either of which of size at least
  $\exp\!\big[ c (\log |A|)^{\delta} \big]$.
  This extends analogous results obtained by Sanders and, respectively, 
  by Croot, Laba and Sisask in the case where the group 
  is $\Z^s$ or $\F_q^n$.
\end{abstract}


\section{Introduction}
\label{sec:introduction}

Our aim in this work is to generalize two 
types of results of additive combinatorics usually 
stated for dense subsets of the integers, 
namely Roth's theorem~\cite{roth} 
and Bourgain's theorem
on long arithmetic progressions in 
sumsets~\cite{bourgainlongaps},
to the case where the sets only have small doubling 
and live in an arbitrary abelian group.
As in previous work of this 
nature~\cite{sanders3aps,ruzsamodelling,solymosi,stanchescu}, 
our motivation is to provide a link between two 
types of additive structure: small doubling on the one hand,
and containment of arithmetic progressions in the
set or its sumset on the other hand.
Since the result we seek is known qualitatively
by the modelling methods of Green and Ruzsa~\cite{greenruzsamodelling},
we focus on the quantitative bounds that may be obtained for it.

Concerning the first topic of Roth's theorem, 
we start by recalling the state-of-the-art bounds,
which we state in the setting of a cyclic group.
Here a $k$-term arithmetic progression 
in an abelian group is defined
as a tuple ${(x_1,\dots,x_k)}$, 
where $x_1,\dots,x_k$ are group elements such that
${x_2-x_1 = \dots = x_k-x_{k-1}}$, 
and we say that it is trivial when $x_1,\dots,x_k$
are all equal, and proper when they are all distinct;
note that when the group has odd order every
nontrivial three-term arithmetic progression is proper.
The breakthrough work of Sanders~\cite{sandersroth} then, 
building on earlier work of Bourgain~\cite{bourgainroth},
has established that given a large enough, odd integer $N$,
every subset of $\Z/N\Z$ of density at least $(\log N)^{-1+o(1)}$
contains a proper three-term arithmetic progression.
Under a density hypothesis, the generalization 
to finite abelian groups is not very challenging:
indeed it can be essentially read out of~\cite{sandersroth}
that any set of density at least $(\log |G|)^{-1+o(1)}$ 
in a finite abelian group $G$ of odd order 
contains a proper three-term arithmetic 
progression.

However, the situation is more complex
when we only assume that the set in question, say $A$,
has small doubling in the sense that $|A + A| \leq K |A|$.
Since subsets of density $\alpha$ of a 
finite abelian group
have doubling at most $K = \alpha^{-1}$, 
this includes the previous situation.
We would then like to show that 
${ K \leq (\log |A|)^{ 1 - o(1)} }$
forces $A$ to contain a proper
three-term arithmetic progression,
which would truely generalize the dense case, 
however this is not not obvious even 
in the case where $A$ is a set of integers.
Indeed the direct approach, which proceeds by combining 
the standard Ruzsa modelling lemma~\cite{ruzsamodelling}
with the bounds for Roth's theorem from~\cite{sandersroth},
only yields an admissible range of ${ K \leq (\log |A|)^{1/4-o(1)} }$.
This is precisely what led Sanders~\cite{sanders3aps} to design a
more subtle approach which, for sets of integers,
yields the range we seek.

\begin{theorem}[Sanders]
\label{thm:san_2xto3apsZ}
	There exists an absolute constant $c > 0$
	such that the following holds. 
	Suppose that $A$ is a finite set of integers such 
	that\footnote{Throughout this introduction,
	we make the tacit assumption that all quantities 
	appearing inside a double logarithm are
	at least $e^e$ in size.}
	\begin{align*}
		|A + A| \leq c (\log |A|) (\log\log |A|)^{-8} 
		\cdot |A|.
	\end{align*}
	Then $A$ contains a proper three-term arithmetic progression.
\end{theorem}

This does not appear explicitely in the literature, 
but follows more or less directly from inserting 
Ruzsa's modelling bound~\cite{ruzsamodelling}
into the argument of~\cite{sanders3aps},
taking also into account the latest bounds
for Roth's theorem~\cite{sandersroth};
we describe this in more detail at the end of the article.
By this procedure, one can actually obtain a version of 
Theorem~\ref{thm:san_2xto3apsZ} for any group with 
good modelling in the sense of~\cite{greenruzsamodelling}.
In the general abelian case, where available modelling arguments
are by necessity much weaker~\cite{greenruzsamodelling}, 
Sanders~\cite{sanders3aps} also improves substantially
on the bounds that would follow from a direct modelling approach.

\begin{theorem}[Sanders]
\label{thm:san_2xto3aps}
	There exists an absolute constant $c > 0$
	such that the following holds.
	Suppose that $A$ is a finite subset 
	of an abelian group such that
	\begin{align*}
		|A + A| \leq c ( \log |A| )^{1/3} (\log\log |A|)^{-1} 
		\cdot |A|.
	\end{align*}
	Then $A$ contains a nontrivial 
	three-term arithmetic progression.
\end{theorem}

Note that the conclusion changed to
yield a nontrivial arithmetic progression only;
we say more on this later. 
The loss in the exponent of $\log |A|$ in comparison 
with the previous case is due to a limitation 
of the results on modelling; indeed
via \cite{greenruzsamodelling} it is only possible 
to Freiman-embed a set $A$ of doubling $K$ into
a finite abelian group where its image has density
$\exp[ -CK^2\log K ]$.
A construction by Green and Ruzsa~\cite{greenruzsamodelling} 
further shows that any modelling result of this type
will feature an exponential loss in $\sqrt{K}$, 
at least if we insist on embedding the whole set.
Fortunately, in a recent major advance 
on the polynomial Freiman-Ruzsa conjecture,
Sanders~\cite{sandersBR} managed to sidestep this issue 
and obtained a correlation result which
may be viewed as another form of modelling.
This result may be applied to our situation
to recover a range of doubling matching
the current bounds for Roth's theorem,
for arbitrary abelian groups;
this is the first observation of this paper.

\begin{theorem}
\label{thm:me_2xto3aps}
	There exists an absolute constant $c > 0$ 
	such that the following holds.
	Suppose that $A$ is a finite subset 
	of an abelian group such that
	\begin{align*}
		|A + A| \leq 
		c ( \log |A| ) ( \log\log |A| )^{-7} 
		\cdot |A|.
	\end{align*}
	Then $A$ contains a nontrivial 
	three-term arithmetic progression.
\end{theorem}

Here we say more on the issue of $2$-torsion,
which was already discussed by Sanders in~\cite{sanders3aps}. 
In general, a set $A$ contains a nontrivial degenerate 
arithmetic progression $(x,y,x)$ if and only if
$A - A$ contains an element of order $2$;
therefore in that case,
Theorems~\ref{thm:san_2xto3aps}~and~\ref{thm:me_2xto3aps} 
give only trivial information.
Obtaining proper progressions in every case
where it is possible (this excludes groups such as $\F_2^n$)
is a thorny issue that has only been successfully adressed
in work of Lev~\cite{lev} and Sanders~\cite{sandersZ4}
in cases where the group rank is not too large; 
here we do not consider this issue.

The second topic we consider is that of 
long arithmetic progressions in sumsets,
initiated by Bourgain~\cite{bourgainlongaps} 
and further developed by Green~\cite{greenapsumsets}.
Basing themselves on a fundamental new technique 
introduced by Croot and Sisask~\cite{CS}, 
these two last authors together with Laba \cite{CLS} 
obtained a remarkable extension of Green's result,
which furthermore already works under a 
small doubling hypothesis.

\begin{theorem}[Croot, Laba, Sisask]
\label{thm:CLS_2xtolongaps}
	There exists an absolute constant $c > 0$ 
	such that the following holds.
	Let $K,L \geq 1$ be parameters,
	and suppose that $A,B$ are finite sets of integers
	such that ${|A + B| \leq K|A|}$ and ${|A + B| \leq L|B|}$.
	Then $A + B$ contains an arithmetic progression of
	length at least
	\begin{align*}
		\exp\bigg[ c \bigg( \frac{\log |A+B|}{K (\log L)^3} \bigg)^{\!1/2} \, \bigg]
		\quad\text{provided}\quad
		K \log^5 ( L \log |A|) \leq c \log |A+B|.
	\end{align*}
\end{theorem}

From the methods of~\cite{CLS}, one can easily deduce
that an analog result holds for subsets $A$ and $B$
of density $\alpha$ and $\beta$ of a finite abelian group,
with $\alpha^{-1}$ and $\beta^{-1}$
in place of $K$ and $L$.
Therefore we focus again on the case of small doubling
in an arbitrary abelian group,
to which the argument of~\cite{CLS} does not
extend as it relies on a two-sets version of 
Ruzsa modelling~\cite{ruzsamodelling}.
The coveted generalization of Theorem~\ref{thm:CLS_2xtolongaps} 
may however be recovered, 
again by using the Bogolyubov-Ruzsa lemma from~\cite{sandersBR},
and establishing this is the second aim of this paper.
Note that in the general abelian setting, we need to 
adapt the type of structure sought to allow for
both cosets of subgroups and arithmetic progressions.

\begin{theorem}
\label{thm:me_2xtolongaps}
	There exists an absolute constant $c > 0$
	such that the following holds.
	Let $K \geq 1$ be a parameter and
	suppose that $A$ is a finite subset of an abelian group
	such that $|A + A| \leq K|A|$.
	Then $A + A$ contains a set, which is either a proper
	arithmetic progression or a coset of a subgroup,
	of size at least
	\begin{align*}
		\exp\bigg[ c\Big( \frac{\log |A|}{K (\log K)^3} \Big)^{1/2} \bigg]
		\quad\text{provided}\quad
		K \leq \frac{c\log |A|}{(\log\log |A|)^5}.
	\end{align*}
\end{theorem}

This recovers Theorem~\ref{thm:CLS_2xtolongaps}
in the symmetric case $A = B$, since
in $\Z$ every nontrivial subgroup is infinite.
We restrict to the symmetric case for simplicity;
it seems feasible to obtain an asymmetric result
of the shape of Theorem~\ref{thm:CLS_2xtolongaps}
from the methods of this paper, however we do not
pursue this here.

Finally, we mention an application of results on
arithmetic progressions in sets of small doubling, 
to the asymptotic size of restricted sumsets.
This application was first observed independently by 
Schoen~\cite{schoen} and Hegyvári et al.~\cite{hegyvarietal} 
in the setting of integers, and later quantitatively 
strengthened by Sanders~\cite{sanders3aps}
in the more general setting of abelian groups.
We write $A \,\wh{+}\, A$ for the set of sums of 
distinct elements of $A$ below.

\begin{corollary}
\label{thm:me_restrictedss}
	Suppose that $A$ is a finite nonempty subset 
	of an abelian group. 
	Then
	\begin{align*}
		|A \,\wh{+}\, A| \geq \big( 1 - (\log |A|)^{-1+o(1)} \big) |A+A|.
	\end{align*}
\end{corollary}

This improves upon the exponent $-\tfrac{1}{3}$ on the logarithm
obtained by Sanders~\cite{sanders3aps} via Theorem~\ref{thm:san_2xto3aps},
since Theorem~\ref{thm:me_2xto3aps} is used instead.
Note that by Behrend's construction~\cite{lyallbehrend}, 
the restricted sumset may have size as low as 
${ ( 1 - e^{-c\sqrt{\log |A|} } ) |A+A| }$ and
therefore the bounds for this problem match those for Roth's theorem
closely.

Finally, we remark that by the finite modelling argument of 
Green and Ruzsa~\cite[Lemma~2.1]{greenruzsamodelling},
it suffices to prove all our results
in the case where the group is finite abelian, 
and therefore we work under that hypothesis
for the rest of the paper.
This concludes our introduction and
we discuss the structure of this paper
in the next section.

\medskip

\textbf{Funding.}
This research was supported by
a \textit{contrat doctoral} from
Université Paris~$7$ and 
by the ANR~Caesar ANR-12-BS01-0011.

\section{Overview}

In this section we sketch the argument 
behind our results and outline the structure of this paper.
We use the symbols $\approx$ and $\gtrsim$
to indicate statements that hold true up to certain 
negligible factors.

The first logical step in the proof 
of Theorem~\ref{thm:me_2xto3aps}
consists in applying the correlation version
of Sanders' Bogolyubov-Ruzsa lemma~\cite{sandersBR}
(Proposition~\ref{thm:br_correl}) to deduce
that a set $A$ of doubling $K$ has density~$\asymp 1/K$ 
in (a translate of) a large Bourgain system $B$, a group-like object
whose properties are recalled in Section~\ref{sec:bsyst}.
The second step is to obtain an efficient 
local version of Roth's theorem 
(Proposition~\ref{thm:roth_localroth}), 
which, roughly saying, asserts that a set $A$ of density 
$\alpha \gtrsim (\log |B|)^{-1}$
in a large Bourgain system $B$
contains many arithmetic progressions,
and therefore a nontrivial one.
This may be applied to the previous system $B$,
for which $|B| \approx |A|$ and $\alpha \asymp 1/K$, 
under the condition $K \lesssim \log |A|$,
thereby establishing Theorem~\ref{thm:me_2xto3aps}.
The local Roth theorem is developed in Section~\ref{sec:roth},
drawing on analytic tools from Section~\ref{sec:local},
and it is combined in the preceding fashion
with the correlation Bogolyubov-Ruzsa lemma
in Section~\ref{sec:3aps}.

To derive Theorem~\ref{thm:me_2xtolongaps}, 
we need to obtain instead a local version
of an almost-periodicity lemma of Croot et al.~\cite{CLS}
(Proposition~\ref{thm:long_almostp}), 
drawing again on the tools of Section~\ref{sec:local}.
This process, carried out in Section~\ref{sec:long}, 
requires a somewhat simpler version of Sanders'
Bogolyubov-Ruzsa lemma (Proposition~\ref{thm:br_cont}) 
which deduces containment of a large Bourgain system 
in the sumset $2A-2A$ from the 
hypothesis that $A$ has small doubling,
and the rest of the argument follows the
strategy of~\cite{CLS}.

Finally, to illustrate some of the above ideas,
we showcase the proof
of Theorem~\ref{thm:me_2xto3aps} in the model
setting of $\F_3^n$, where the proof of Sanders' 
Bogolyubov-Ruzsa lemma~\cite{sandersBR}
simplifies substantially.
As an added benefit, the formidable bounds of
Bateman and Katz~\cite{batemankatz} 
for caps in $\F_3^n$
yield a larger admissible range of doubling
in this setting.
The notation used in the proof is introduced
in Section~\ref{sec:notation}.

\begin{theorem}
\label{thm:me_2xto3aps_ff}
	There exist positive absolute constants 
	$c$ and $\eps$ such that the following holds.
	Suppose that $A$ is a subset of $\F_3^n$ such that
	\begin{align*}
		|A + A| \leq c (\log |A|)^{1 + \eps} \cdot |A|.
	\end{align*}
	Then $A$ contains a proper three-term arithmetic progression.
\end{theorem}

\begin{proof}
Write $K = |A + A| / |A|$, so that
we are assuming that $K \leq c (\log |A|)^{1 + \eps}$.
The proof of~\cite[Proposition~6.1]{greenruzsamodelling}
readily adapts to $\F_3^n$, and shows that
$A$ is Freiman-isomorphic to a subset
of doubling $K$ and density at least $K^{-4}$
of another finite field $\F_3^m$,
which we identify with $A$ from now on.
By examining the proof 
of~\cite[Theorem~A.1]{sandersBR},
which works equally well in $\F_3^m$,
one may deduce that there exist a 
measure $\mu$ and a subspace
$V$ of $\F_3^m$ of codimension
at most $C(\log K)^4$ such that
\begin{align*}
	\langle 1_A \ast \mu_V \ast \mu_{A + A} \ast \mu, \mu_A \rangle_{L^2} 
	\geq \tfrac{1}{2} \mu_G(A) / \mu_G(A+A)  . 
\end{align*}
By the definition of $K$, and upon applying 
Hölder's and Young's inequalities, we obtain
\begin{align*}
	\tfrac{1}{2K}
	&\leq \langle 1_A \ast \mu_V \ast \mu_{A + A} \ast \mu, \mu_A \rangle_{L^2} \\
	&\leq \| 1_A \ast \mu_V \ast \mu_{A + A} \ast \mu \|_{\infty} \| \mu_A \|_{L^1} \\
	&\leq \| 1_A \ast \mu_V \|_\infty .
\end{align*}
Therefore we may find $x$ such that $A' = (A - x) \cap V$
has density at least $\tfrac{1}{2K}$ in $V$.
Since $V$ has codimension at most $C (\log K)^4$,
it has size at least $|G|^{1/2}$ in our range of $K$.
Applying~\cite[Theorem~1.1]{batemankatz} to $A'$, 
we are then ensured to find a proper three-term 
arithmetic progression in $A'$ provided
\begin{align*}
	\tfrac{1}{2K} \geq C (\log |V|)^{-(1+\eps)}
\end{align*}
and this concludes the proof
since $\log |V| \asymp \log |A| $.
\end{proof}

\section{Notation}
\label{sec:notation}

In this section we introduce the notation
used throughout the article.

\textit{Ambient group.}
We let $G$ denote a fixed, finite abelian group.
The arguments of later sections all take place
in this group unless otherwise stated.

\textit{$\Z$-actions.}
The group $G$ is naturally equipped with 
a structure of $\Z$-module, 
and we let $k \cdot x$ denote the action 
of a scalar $k \in \Z$ on an element $x \in G$.
For a subset $X$ of $G$ and a subset $I$ of $\Z$,
we further write
\begin{align*}
	k \cdot X = \{ k \cdot x : x \in X \}
	\quad\text{and}\quad
	I \cdot x = \{ k \cdot x : k \in I \}.
\end{align*}
Note that $\cdot$ is also used in other places
for the regular multiplication of complex numbers, 
however it should be clear from the context which one is meant.

\textit{Functions.}
We define the averaging operator over a subset $X$ of $G$, 
which acts on the space of functions $f : G \rightarrow \C$, 
by ${\E_X f = |X|^{-1} \sum_{x \in X} f(x)}$,
and we write $\E_{x \in X} f(x)$ when we want to
keep the variable explicit.
It is also convenient to introduce 
the operator of translation on a function $f$
defined by $\tau_x f (u) = f(x+u)$ for all $x,u \in G$.
We furthermore define the support of $f$ as 
$\Supp(f) = \{ x \in G : f(x) \neq 0 \}$.
On the physical space, we use the normalized counting measure 
so that for functions ${f,g : G \rightarrow \C}$, we let
\begin{align*}
	&\text{($L^p$-norm)} &
	\| f \|_{L^p} &= (\E_G \, |f|^p)^{1/p} , &
	& \\
	&\text{(Scalar product)} &
	\langle f, g \rangle_{L^2} &= \E_G \, f \bar{g} , &
	& \\
 	&\text{(Convolution)} &
 	f \ast g (x) &= \E_{y \in G} f(y) g(x - y) &
	&\forall x \in G.
\end{align*}
We occasionally write $\| f \|_p $ for $\| f \|_{L^p}$,
and we let $f^{(\ell)}$ denote the
convolution of $f$ with itself $\ell$ times.

\textit{Measures.}
We identify measures $\mu$ on $G$ with
functions ${\mu : G \rightarrow \R_+}$ via the
identity $\mu( \{ x \} ) = |G|^{-1} \mu(x)$, so that
$\mu (E) = \langle 1_E , \mu \rangle_{L^2}$
for every subset $E$ of $G$.
We only consider probability measures;
in other words, we always assume that
$\| \mu \|_{L^1} = 1$.
We write $\mu_A$ for the measure defined by 
$\mu_A (E) = |E \cap A| / |A|$ for every set $E$, 
which under our identification
corresponds to the function 
$\mu_A = \mu_G(A)^{-1} 1_A$.

\textit{Fourier transform.}
The Fourier transform over finite abelian groups
is now a standard tool of additive combinatorics.
It is very well explained for example in~\cite{hatami},
and here we only recall its main properties.

Write $\U$ for the unit circle,
then the dual group $\dual$ is
defined as the set of morphisms from $G$ to $\U$,
called characters,
and the Fourier transform of a function 
$f : G \rightarrow \C$
is defined by ${\wf (\gamma) = \langle f , \gamma \rangle_{L^2}}$ 
at every character $\gamma$.
We write $(f)^\wedge$ for the Fourier transform of $f$
when $f$ has a complicated expression.

We define the summation operator over a subset $\Delta$ of $\dual$,
which acts on the space of functions $F : \dual \rightarrow \C$, by
${\sum_{\Delta} F = \sum_{\gamma \in \Delta} F(\gamma)}$.
On the Fourier space, we use the counting measure
so that for functions $F,G : \dual \rightarrow \C$, we let
\begin{align*}
	&\text{($\ell^p$-norm)} &
	\|F\|_{\ell^p} &= \big(\! \ssum_{\dual} |F|^p \big)^{1/p} \! , & 
	&\phantom{\text{($\ell^p$-norm)}} \\
	&\text{(Scalar product)} &
	\langle F,G \rangle_{\ell^2} &= \ssum_{\dual}  F \, \cj{G}. & 
	&\phantom{\text{($\ell^p$-norm)}}
\end{align*}
The three classic formul\ae\ of harmonic analysis 
then read as follows:
\begin{align*}
	&\text{(Fourier inversion)} &
	f &= \ssum_{\dual} \wf(\gamma) \gamma, &
	&\phantom{\text{(Fourier inversion)}} \\
	&\text{(Parseval formula)} &
	\langle f, g \rangle_{L^2} &= \langle \wf, \wg\, \rangle_{\ell^2}, &
	&\phantom{\text{(Fourier inversion)}} \\
	&\text{(Convolution identity)} &
	( f \ast g )^\wedge &= \wf \cdot \wg \,. &
	&\phantom{\text{(Fourier inversion)}}
\end{align*}

\textit{Other.}
We let $c$ and $C$ denote absolute positive constants,
which may take different values at each occurence.
Given nonnegative functions $f$ and $g$, we let
$f = O(g)$ or $f \ll g$ indicate the fact that
there exists a constant $C$ such that $f \leq Cg$, 
and we let $f = \Theta(g)$ or $f \asymp g$
indicate that $f \ll g$ and $g \ll f$ hold simultaneously.
We also write $\ell(x) = \log(e/x)$ for $x \geq 1$,
since this quantity arises often in our computations.
Note finally that in many occurences of 
logarithms throughout the paper,
one should replace $\log x$ by $\log ex$ for the
results to be formally correct in all ranges of parameters;
we leave this as a mental task to the reader 
to alleviate the notation.
Other notation in this paper is introduced
in the relevant section as needed.

\section{Bourgain systems}
\label{sec:bsyst}

In this section we recall the theory of Bourgain systems,
which was introduced by Green and Sanders~\cite{greensanders}
as a generalization of the Bohr set technology
of Bourgain~\cite{bourgainroth}.
In a sense these systems are the most general class 
of sets for which the strategy 
of density increment on Bohr sets,
pioneered by Bourgain~\cite{bourgainroth}, 
may be carried out.
What is needed for such an undertaking is for the set 
to behave approximately like a $d$-dimensional ball
with respect to dilation,
as axiomatized in the following definition.

\begin{definition}[Bourgain system]
\label{thm:bsyst_def}
	A Bourgain system of dimension $d$ is a family of sets
	${\mB = (B_\rho)_{\rho > 0}}$, 
	where $B_\rho$ are subsets of $G$ such that, 
	for all positive $\rho$ and $\rho'$,
	\begin{align*}
		&\text{(containment of $0$)} &
		0 &\in B_{\rho} &&\\
		&\text{(symmetry)} &
		- B_{\rho} &= B_{\rho} &&\\
		&\text{(nesting)} &
		B_{\rho} &\subset B_{\rho'} & &\text{if}\quad \rho \leq \rho' \\
		&\text{(additive closure)} &		
		B_{\rho} + B_{\rho'} &\subset B_{\rho+\rho'} &&\\
		&\text{($2^d$-covering)} &
		\exists X_{\rho} : B_{2\rho} &\subset X_{\rho} + B_{\rho} &
	 	&\text{and}\quad |X_{\rho}| \leq 2^d.
	\end{align*}
	We write $B = B_1$, and we define
	the density of $\mathcal{B}$ as $b = |B| / |G|$.
\end{definition}

We let the sets $B_\rho$, and sometimes also 
the dimension $d$ and the density $b$, 
be defined implicitely whenever we 
introduce a Bourgain system $\mB$. 
We now describe two important classes of Bourgain systems:
Bohr sets and coset progressions.
To define the former, we consider 
the multiplicative analog $\| \cdot \|_{\U}$ on the unit circle
of the usual pseudo-norm~$\| \cdot \|_{\T} = d( \,\cdot , \Z )$ 
on the torus, defined by ${ \| e(\theta) \|_{\U} = \| \theta \|_{\T} }$
for every $\theta \in \T$.

\begin{definition}[Bohr set]
\label{thm:bset_def}
	Suppose that $\Gamma \subset \dual$ and $\delta > 0$.
	The Bohr set of frequency set $\Gamma$ 
	and radius $\delta$ is 
	\begin{align*}
		B = B(\Gamma,\delta) = \{ x \in G : \| \gamma(x) \|_{\U} \leq \delta \}.
	\end{align*}
	The dimension of $B$ is $d = |\Gamma|$.
	We define the dilate of $B$ by $\rho > 0$ as 
	the set $B_{\rho} = B(\Gamma,\rho\delta)$,
	and the Bohr system induced by $B$ as the system 
	$\mathcal{B} = (B_\rho)_{\rho > 0}$.
\end{definition}

The usual bounds for the size and growth of a Bohr set 
allow us to quickly estimate the dimension 
and density of the Bourgain system it induces.

\begin{lemma}
\label{thm:bset_dimdensity}
	The system $\mathcal{B}$ induced by a Bohr set
	of dimension $d$ and radius $\delta \leq 1$
	is a Bourgain system $\mathcal{B}$
	of dimension at most $6d$ and density at least $\delta^d$.
\end{lemma}

\begin{proof}
The first four properties of 
a Bourgain system are easy to check.
Further, by three applications of~\cite[Lemma~4.20]{taovu} we obtain 
$|B_{4\rho}| \leq 2^{6d} |B_{\rho/2}|$, and therefore
by Ruzsa's covering lemma we may find a set $X_\rho$ such that
\begin{align*}
	B_{2\rho} \subset X_\rho + B_{\rho/2} - B_{\rho/2} \subset X_\rho + B_{\rho}
\end{align*}
and $|X_\rho| \leq |B_{2\rho} + B_{\rho/2}| / |B_{\rho/2}| \leq 2^{6d}$.
Working through the argument in that reference, 
one could extract a better bound $2^{2d}$, but this
would not affect our end results much.
The bound on the density may be read directly 
from~\cite[Lemma~4.20]{taovu}.
An alternate reference for these estimates 
is~\cite[Section~5]{meexpobourgain}.
\end{proof}

In our definition of a coset progression,
we write $[x,y]_\Z = \{ n \in \Z : x \leq n \leq y \}$
for reals $x \leq y$.

\begin{definition}[Coset progression]
\label{thm:cprog_def}
	Let $L \in \R_{+}^{d}$ and $\omega \in G^d$ where $d\geq 1$, 
	and let $H$ be a subgroup of $G$.
	The coset progression of dimension $d$ determined by
	$L,\omega,H$ is 
	\begin{align*}
		M = M(L,\omega,H) = 
		[-L_1,L_1]_{\Z} \cdot \omega_1 + \dotsb + [-L_d,L_d]_{\Z} \cdot \omega_d + H.
	\end{align*}
	We define the dilate of $M$ by $\rho >0$ as 
	$M_{\rho} = M(\rho L,\omega,H)$, 
	and the coset progression system induced by $M$ as
	the system $\mM = (M_\rho)_{\rho > 0}$.
\end{definition}

The dimension of the Bourgain system induced by 
a coset progression may be estimated by a simple 
covering argument.

\begin{lemma}
\label{thm:cprog_systdim}
	The system $\mM$ induced by a
	$d$-dimensional coset progression $M$ is
	a Bourgain system of dimension at most $3d$.
\end{lemma}

\begin{proof}
	It is again rather simple to derive the first
	four properties of a Bourgain system for $\mM$,
	and we now concern ourselves with the fifth.
	The dilate of $M$ by $\rho > 0$ is
	\begin{align*}
		M_\rho = 
		[-\rho L_1,\rho L_1]_{\Z} \cdot \omega_1 + \dotsb + [- \rho L_d, \rho L_d]_{\Z} \cdot \omega_d + H .
	\end{align*}
	To obtain the covering property, first 
	observe that for any $k \in \N_{\geq 0}$,
	one may cover the interval $[-k,k]_{\Z}$ 
	by three translates of $[ - \tfrac{k}{2} , \tfrac{k}{2} ]_{\Z}$
	(this is sharp for $k$ odd), 
	and that this still holds for any real $k \geq 0$.
	Therefore, for every $1 \leq i \leq d$, we may find 
	a set $T_i$ with $|T_i| \leq 3$ such that
	$[-2\rho L_i,2\rho L_i]_{\Z} \subset T_i + [-\rho L_i,\rho L_i]_{\Z}$.
	Consequently, for any $\rho > 0$ we have a covering
	\begin{align*}
		M_{2\rho} 
		\subset \bigcup_{t \in T_1 \times \dotsb \times T_d}
		( t_1 \cdot \omega_1 + \dotsb + t_d \cdot \omega_d + M_{\rho} )
		= X_\rho + M_\rho
	\end{align*}
	for a certain set $X_\rho$ of size at most $|T_1| \cdots |T_d| \leq 3^d$.
\end{proof}

With these examples covered, we now
work exclusively within the framework of Bourgain systems.
We start by defining a few basic operations 
on these systems.

\begin{lemma}[Dilation]
\label{thm:bsyst_dilation}
	Suppose that $\lambda \in (0,1]$ and that
	$\mB$ is a Bourgain system of dimension $d$ and density $b$.
	Then the dilated system
	$\mB_\lambda = \big( B_{\lambda\rho} \big)_{\rho > 0}$
	is a Bourgain system of dimension at most $d$ 
	and density at least $(\lambda/2)^d \cdot b$.
\end{lemma}

\begin{proof}
Let $\lambda \in (0,1]$, and choose $k \geq 0$
such that $2^{-(k+1)} < \lambda \leq 2^{-k}$.
By the covering property
of Definition~\ref{thm:bsyst_def},
we have $|B_{\rho}| \leq 2^d |B_{\rho/2}|$ for every $\rho > 0$,
from which it follows by iteration that
$ |B| \leq 2^{(k+1)d} |B_{1/2^{k+1}}| \leq (2/\lambda)^d |B_\lambda|$.
That $\mB_\lambda$ is a $d$-dimensional Bourgain system is obvious,
and the bound on the density follows from 
the previous computation.
\end{proof}


\begin{definition}[Sub-Bourgain systems]
	Suppose that $\mB$ and $\mB'$ are two Bourgain systems.
	We say that $\mB$ is a sub-Bourgain system of $\mB'$,
	and we write $\mB \leq \mB'$, 
	when $B_\rho \subset B'_\rho$ for all $\rho > 0$.
	For $\lambda \in (0,1]$, we also write 
	$\mB \leq_\lambda \mB'$ when $\mB \leq \mB_{\lambda}'$.
\end{definition}

The properties of an intersection of Bourgain systems
were derived in~\cite[Lemma~3.4]{sanders3aps}, whose
proof we reproduce here for completeness.

\begin{lemma}[Intersection]
\label{thm:bsyst_intersection}
	Suppose that $\mB^{(1)},\dots,\mB^{(k)}$ are Bourgain systems
	of dimensions $d_1,\dots,d_k$ and densities $b_1,\dots,b_k$.
	Then the intersection system
	\begin{align*}
		\mB_1 \wedge \dots \wedge \mB_k 
		= ( B_\rho^{(1)} \cap \dots \cap B_\rho^{(k)})_{\rho > 0}
	\end{align*}
 	is a Bourgain system of dimension at most
	$2(d_1 + \dotsb + d_k)$ and of
	density at least ${ 4^{-(d_1 + \dotsb + d_k)} b_1 \cdots b_k }$.
\end{lemma}

\begin{proof}
The first four properties of a Bourgain
system are again easy to check, and we
now consider the covering property.
Let $\rho > 0$. 
For each $1 \leq i \leq k$,
apply the covering property of $\mB^{(i)}$ twice 
to obtain a set $T_i$ of size at most $4^{d_i}$ such that
$B_{2\rho}^{(i)} \subset T_i + B_{\rho/2}^{(i)}$.
Distributing intersection over union, we have then
\begin{align*}
	\ccap_{1 \leq i \leq d} B_{2\rho}^{(i)}
	= \ccup_{(t_1,\dots,t_k) \in T_1 \times \cdots \times T_k}
	\ccap_{1 \leq i \leq k} \big( t_i + B_{\rho/2}^{(i)} \big).
\end{align*}
Now pick an element $x(t)$ in each nonempty intersection
$\bigcap_i ( t_i + B_{\rho/2}^{(i)} )$.
Then for each element $x$ of $\bigcap_i B_{2\rho}^{(i)}$,
we may find an element $t \in \prod_i T_i$ such that
\begin{align*}
	x - x(t) 
	\in \ccap_i \big( B_{\rho/2}^{(i)} - B_{\rho/2}^{(i)} \big) 
	\subset \ccap_i B_{\rho}^{(i)}.
\end{align*}
This yields the desired covering with $X_{\rho}$ 
defined as the set of all $x(t)$.

To estimate the density of the intersection, 
first apply Ruzsa's covering lemma for each $1 \leq i \leq k$
to obtain a covering of the form
\begin{align*}
	G \subset T_i + B_{1/4}^{(i)} - B_{1/4}^{(i)} \subset T_i + B_{1/2}^{(i)}
\end{align*}
where $T_i$ is a set of size $|T_i| \leq 4^{d_i} b_i^{-1}$.
From $G \subset \cap_i \,( T_i + B_{1/2}^{(i)} )$, it follows that
\begin{align*}
	G 
	= \ccup_{(t_1,\dots,t_k) \in T_1 \times \cdots \times T_k}
	\ccap_{1 \leq i \leq k} \big( t_i + B_{1/2}^{(i)} \big) 
	= \ccup_{t \in T_1 \times \cdots \times T_k} A(t)
\end{align*}
where $A(t)$ are sets satisfiying
$ A(t) - A(t) \subset \bigcap_i B^{(i)} $.
By the pigeonhole principle, 
we may also find a point $t$ such that
\begin{align*}
	|A(t)| \geq \frac{|G|}{|T_1| \cdots |T_k|} \geq 
	4^{ - (d_1 + \dotsb + d_k)} b_1 \cdots b_k |G|,
\end{align*}
which yields the desired density estimate
since $|A(t) - A(t)| \geq |A(t)|$.
\end{proof}

We consider one last operation on Bourgain systems;
since it is so simple we leave it as an exercise 
to the reader.

\begin{lemma}[Homomorphic image]
\label{thm:bsyst_homimage}
	Suppose that $\mB$ is 
	a Bourgain system of dimension $d$,
	and $\phi$ is an endomorphism of $G$.
	Then the image system 
	$\phi(\mB) = \big( \phi(B_\rho) \big)_{\rho > 0}$
	is a Bourgain system of dimension at most $d$.
\end{lemma}

Finally, we recall the essential notion of regularity
introduced by Bourgain~\cite{bourgainroth} 
for Bohr sets, and which has a natural analogue
for Bourgain systems.
We let\footnote{These precise constants, featured in
subsequent lemmas, are derived
in~\cite[Section~6]{meexpobourgain}.}
$C_0 = 2^5$ and $C_1 = 2^6$ in what
follows for definiteness, although the exact values
are unimportant.

\begin{definition}[Regular Bourgain system]
\label{thm:reg_def}
	We say that a Bourgain system $\mB$ of dimension $d$ is regular when,
	for every $|\rho| \leq \tfrac{1}{C_0 d}\,$,
	\begin{align*}
		1 - C_0 |\rho| d \leq \frac{|B_{1+\rho}|}{|B|} \leq 1 + C_0 |\rho| d.
	\end{align*}
\end{definition}

In practice one can always afford
to work with regular Bourgain systems,
as is the case with Bohr sets,
via~\cite[Proposition~3.5]{sanders3aps}
which we now quote.

\begin{lemma}
\label{thm:reg_regularizing}
	Suppose that $\mB$ is a Bourgain system.
	Then there exists 
	$\lambda \in \big[ \tfrac{1}{2},1 \big]$ 
	such that $\mB_{\lambda}$ is regular.
\end{lemma}

The regularity computations in subsequent
sections rely on the following $L^1$ estimate.

\begin{lemma}
\label{thm:reg_averaging}
	Suppose that $\mB$ is a regular Bourgain system of
	dimension $d$ and $\mu$ is a measure on $G$ with support in $B_\rho$,
	where $0 < \rho \leq \tfrac{1}{C_1 d}$.
	Then
	\begin{align*}
		\| \mu_B \ast \mu - \mu_B \|_{L^1} \leq C_1 \rho d .
	\end{align*}
\end{lemma}

\begin{proof}
	For every $y \in B_\rho$, the function
	$\mu_{y + B} - \mu_B$ has support in
	$B_{1+\rho} \smallsetminus B_{1-\rho}$,
	so that
	\begin{align*}
		\| \mu_{y + B} - \mu_B \|_{L^1} 
		\leq \frac{|B_{1+\rho}| - |B_{1-\rho}|}{|B|}
		\leq 2 C_0 \rho d.
	\end{align*}
	Averaging over $y \in G$ with weights $\mu(y)$, 
	and using the triangle inequality, 
	we recover the desired estimate.
\end{proof}

\section{Spectral analysis on Bourgain systems}
\label{sec:local}

This section is concerned with collecting
all the analytic information we need about
the large spectrum of the indicator functions of
certain sets.
The main task is to obtain a large structured set
on which all characters of the large spectrum
take values close to $1$,
since such a set may be later used for purposes 
of a density-increment-based iteration, 
or to locate long arithmetic progressions.

When considering indicator functions of subsets of Bohr sets, 
the information we seek is provided by the spectral analysis 
developed by Sanders~\cite{sandersother},
and the aim of this section is therefore 
to obtain a similar analysis for Bourgain systems.
Note that such a process was already carried out 
in the earlier article~\cite{sanders3aps},
however we benefit here from
the more efficient analysis of the local spectrum 
from~\cite{sandersother}.
To be specific, there is now a local analog of 
Chang's bound~\cite[Lemma~4.6]{sandersother}
which supersedes the earlier local analog
of Bessel's inequality~\cite[Proposition~4.4]{sanders3aps}.
We now give the precise statements,
and in that regard it is useful to
recall the following definitions.

\begin{definition}[Annihilation]
\label{thm:local_annihdef}
	Let $\nu \in (0,2]$ be a parameter,
	and suppose that $T$ is a subset of $G$
	and $\Delta$ is a subset of $\dual$.
	We say that $\Delta$ is $\nu$-annihilated by $T$ when
	\begin{align*}
		| 1 - \gamma (t)| \leq \nu
		\quad\text{for all $t \in T$ and $\gamma \in \Delta$.}
	\end{align*}
	When $\mB$ is a Bourgain system,
	we say that it $\nu$-annihilates $\Delta$ 
	when $B$ does.
\end{definition}

The quantity we seek to annihilate is then the following.

\begin{definition}[Large spectrum]
	Suppose that $\eta \in (0,1]$ be a parameter
	and $f : G \rightarrow \C$ is a function.
	The $\eta$-large spectrum of $f$ is the level set of $\dual$ defined by
	\begin{align*}
		\Spec_\eta(f) = \{ \, |\wf| \geq \eta \| f \|_{L^1} \}.
	\end{align*}
\end{definition}

We also need to recall one piece of terminology
from~\cite[Section~4]{sandersother}, which is only used in this section.
Write $\D$ for the unit disk,
and let $\mu$ be any measure on $G$.
Given a parameter ${\theta \in (0,1]}$, 
we say that a set $\Lambda$ of characters is 
{$(\theta,\mu)$-dissociated} when,
for every function ${\omega: \Lambda \rightarrow \D}$, we have
\begin{align*}
	\int \prod_{\lambda \in \Lambda} 
	\Big( 1 + \re[ \omega( \lambda ) \lambda ] \Big) 
	\mathrm{d}\mu \leq e^\theta,
\end{align*}
and when $\theta = 1$ we simply say that $\Lambda$ is $\mu$-dissociated.
We may now quote two lemmas of local spectral analysis 
from~\cite{sandersother}, with minor tweaks in both cases.

\begin{lemma}[Local Chang bound]
\label{thm:local_chang}
	Let $\eta \in (0,1]$ be a parameter,
	and suppose that $B$ is a subset of $G$
	and $X$ is a subset of $B$ of density $\tau$.
	Then every $\mu_B$-dissociated subset of
	$\Spec_\eta (\mu_X)$ has size 
	at most $C \eta^{-2} \log \tau^{-1}$.
\end{lemma}

\begin{proof}
This is~\cite[Lemma~4.6]{sandersother}, specialized
to the case where $f = \mu_X$ and $\mu = \mu_B$, so that
with the notation from there
$L_{\mu_X,\mu_B} = \tau^{-1/2}$.
\end{proof}

\begin{lemma}[Annihilating locally dissociated sets]
\label{thm:local_dissocann}
	Let $\nu \in (0,1]$ be a parameter.
	Suppose that $\mB$ is a regular Bourgain system,
	$\Delta$ is a set of characters, and
	$m$ is the size of the largest $\mu_B$-dissociated
	subset of $\Delta$, or $1$ if there is no such subset.
	Then there exists a Bohr set $\wt{B}$ of dimension
	at most $m$ and radius equal to $c / m$ such that
	$\Delta$ is $\nu$-annihilated by
	$B_{c \nu / d^2 m} \cap \wt{B}_\nu$.
\end{lemma}

\begin{proof}
This is~\cite[Lemma~6.3]{sandersother} with $\eta = 1$
and $m = \max(k,1)$, and two minor tweaks:
$\mB$ is a Bourgain system instead of a Bohr set
and a few changes of variables have been effected.
Since the proof requires only a regularity estimate 
of the type of Lemma~\ref{thm:reg_averaging}, 
the generalization to Bourgain systems is immediate.
\end{proof}

As usual these two ingredients combine to
show that the large spectrum of a dense subset
of a Bourgain system
may be efficiently annihilated.
Before carrying this out, we introduce a last
definition which serves to simplify our 
technical statements. 

\begin{definition}
\label{thm:local_bsystcontrol}
	Let $m \geq 1$ be a parameter
	and suppose that $\mB$ is a Bourgain system.
	We say that $\mB$ is 
	{$m$-controlled} when
	it has dimension at most $m$ and density at
	least $\exp[ - Cm\log m ]$.
\end{definition}

We are now ready to introduce
the main technical tool of this paper.
Recall that $\ell(x)$ stands for $\log(e/x)$ here
and throughout the article.

\begin{proposition}[Local spectrum annihilation]
\label{thm:local_specann}
	Let $\eta,\nu \in (0,1]$ be parameters.
	Suppose that $\mB$ is a regular Bourgain system
	and $X$ is a subset of $B$ of relative density $\tau$.
	Then $\Spec_\eta (\mu_X)$ is {$\nu$-annihilated} by 
	a regular Bourgain system of the form 
	\begin{align*}
		\mB_{c\nu / d^2 m} \wedge \wmB_\nu
		\quad\text{where}\quad
		m \leq C \eta^{-2} \ell(\tau) 
	\end{align*}	
	and $\wmB$ is an $m$-controlled Bourgain system.
\end{proposition}

\begin{proof}
Let $m$ denote the size of the largest
$\mu_B$-dissociated subset of $\Spec_\eta (\mu_X)$,
or $1$ when there is no such set.
By Lemma~\ref{thm:local_chang}, we have
$m \leq C\eta^{-2} \ell (\tau)$.
By Lemma~\ref{thm:local_dissocann}, we also know
that $\Spec_\eta (\mu_X) $ is {$\nu$-annihilated} by
a regular Bourgain system  
$\cmB \coloneqq \mB_{c\nu / d^2 m} \wedge \wmB_\nu$,
where $\wmB$ is the Bourgain system induced
by a Bohr set of dimension $d \leq m$ and radius
$\delta = c / m$.
By Lemma~\ref{thm:reg_regularizing}, 
we may further ensure that $\cmB$ is regular up to
dilating it by a factor $\asymp 1$,
which does not affect the shape of
the above intersection except  
in the value of the constants.
By Lemma~\ref{thm:bset_dimdensity},
we also see that $\wmB$ has dimension at most $6m$
and density at least $\exp[-C m \log m]$,
so that the result follows by replacing $6m$
with $m$ and adapting the constants.
\end{proof}

\section{Roth's theorem for Bourgain systems}
\label{sec:roth}

This section is concerned with a local
version of Roth's theorem~\cite{roth},
first considered by Sanders~\cite{sanders3aps},
which applies to dense subsets of a Bourgain system.
Since the pioneering work of Bourgain~\cite{bourgainroth},
modern proofs of Roth's theorem~\cite{sandersother,sandersroth} 
all share the same global structure and proceed 
by an iteration on subsets of Bohr sets.
An important observation made in~\cite{sanders3aps} is
that this iteration may be initialized inside a 
certain Bohr set instead of the whole group,
and further that one may perform the same iteration
on Bourgain systems in place of Bohr sets.

However the quantitative estimates obtained in~\cite{sanders3aps}
correspond roughly in strength to a range of 
$\alpha \gtrsim (\log N)^{-1/3}$ in Roth's theorem, 
while the best-known range, also by Sanders~\cite{sandersroth}, 
is now $\alpha \gtrsim (\log N)^{-1}$.
Conceptually, there is no obstacle in
obtaining this better quantitative dependency 
with Bourgain systems,
and for the same local initialization,
however on a technical level 
it is not entirely straightforward 
as most density-increment statements 
then take a different shape.
We carry out this process in this section;
since it is not the right place here to
present the whole argument of~\cite{sandersroth},
we only include the main structural results we need from it
and indicate the changes that need to be done to other.
Unfortunately, this means that the reader 
needs either to be conversant with~\cite{sandersroth},
or to read this section conditionally on
Proposition~\ref{thm:roth_it2scales} below.
What we obtain eventually is
the following quantitative 
improvement of~\cite[Theorem~5.1]{sanders3aps}.

\begin{proposition}[Local Sanders-Roth theorem]
\label{thm:roth_localroth}
	Suppose that $\mB$ is a regular Bourgain system
	and $A$ is a subset of $B$ of relative density $\alpha$
	such that $A - A$ contains no element of order $2$.	
	Then
	\begin{align*}
		\langle 1_A \ast 1_A , 1_{2 \cdot A} \rangle_{L^2}
		\geq \exp\big[ - C (\alpha^{-1} + d) \ell(\alpha)^6 \ell(\alpha/d) \big] \cdot b^2.
	\end{align*}
\end{proposition}

We make a brief comment here
on the shape of the above proposition.
The three-term arithmetic progressions
contained in a set $A$ are precisely the
triples $(x,y,z)$ of $A^3$
such that $x + z = 2 \cdot y$.
The assumption on $A$ shows that
the change of variables $y \mapsto 2 \cdot y$ 
is injective on $A$, 
from which we see that the total number of such progressions
is equal to $\langle 1_A \ast 1_A , 1_{2 \cdot A} \rangle_{L^2} \cdot |G|^2$.
We invite the reader to keep this observation in mind, 
as it is used implicitely in later arguments.

We now present our modified version of 
the argument of~\cite{sandersroth}.
To begin with, we reconstitute the $L^2$ density-increment 
strategy entirely as it takes a different form for Bourgain systems, 
which determines the shape of iterative statements.
The following lemma is the usual argument that allows one to pass
from large energy of the Fourier transform 
over a character set, to a density increment on any set 
annihilating those characters.

\begin{lemma}
\label{thm:roth_L2tocorrel}
	Let $\rho,\kappa \in (0,1]$ be parameters.
	Suppose that $\mB$ is a regular Bourgain system,
	$A$ is a subset of $B$ of relative density $\alpha$,
	$T$ is a subset of $B_\rho$ and
	$\Delta$ is a set of characters.
	Assume also that $\rho \leq c \kappa \alpha / d$ and 
	write $f_A = 1_A - \alpha 1_B$.
	Then if 
	\begin{align*}
		\sum_\Delta |\wf_A|^2 \geq \kappa \alpha^2 b
		\qquad\text{and $\Delta$ is $\tfrac{1}{2}$-annihilated by $T$},
	\end{align*}
	we have $\| 1_A \ast \mu_T \|_{\infty} 
	\geq ( 1 + 2^{-3} \kappa )\alpha$.
\end{lemma}

\begin{proof}
For every character $\gamma \in \Delta$ we
know that $|1 - \gamma| \leq 1/2$ on $T$, and therefore 
$|\wh{\mu_T}(\gamma) - 1| 
\leq \E_T | 1 - \gamma | \leq \tfrac{1}{2}$
and $|\wh{\mu_T}(\gamma)| \geq \tfrac{1}{2}$.
Inserting this into the energy lower bound, 
we have, via Parseval,
\begin{align*}
	\tfrac{1}{4} \kappa \alpha^2 b
	&\leq \ssum_{\dual} |\wf_A|^2 |\wh{\mu}_T|^2 \\ 
	&= \langle f_A \ast \mu_T , f_A \ast \mu_T \rangle_{L^2}.
\end{align*}
Expanding this scalar product,
and with the help of Lemma~\ref{thm:reg_averaging}, we obtain
\begin{align*}
	\tfrac{1}{4} \kappa \alpha^2 b 
	&\leq \| 1_A \ast \mu_T \|_2^2
	- 2 \alpha \, \langle 1_A \ast \mu_T , 1_B \ast \mu_T \rangle_{L^2}
	+ \alpha^2 \, \langle 1_B \ast \mu_T , 1_B \ast \mu_T \rangle_{L^2} \\
	&= \| 1_A \ast \mu_T \|_2^2
	- 2\alpha b \, \langle 1_A , \mu_B \ast \mu_T \ast \mu_{-T} \rangle_{L^2}
	+ \alpha^2 b \, \langle 1_B , \mu_B \ast \mu_T \ast \mu_{-T} \rangle_{L^2} \\
	&= \| 1_A \ast \mu_T \|_2^2 
	- \big( 1 + O\big( \tfrac{\rho d}{\alpha} \big) \big) \alpha^2 b.
\end{align*}
Choosing $\rho \leq c \kappa \alpha / d$, we have then
\begin{align*}
	( 1 + 2^{-3} \kappa) \alpha^2 b
	&\leq \| 1_A \ast \mu_T \|_2^2 \\
	&\leq \| 1_A \ast \mu_T \|_\infty \| 1_A \ast \mu_T \|_1 \\
	&= \| 1_A \ast \mu_T \|_\infty \cdot \alpha b.
\end{align*}
Dividing both sides by $\alpha b$ concludes the proof.
\end{proof}

As usual this may be combined with 
a statement on the local annihilation of the large spectrum,
such as Proposition~\ref{thm:local_specann},
to recover an $L^2$-density increment lemma.

\begin{proposition}[$L^2$ density-increment]
\label{thm:roth_L2incr}
	Let $\kappa,\eta \in (0,1]$ be parameters.
	Suppose that $\mB$, $\dmB$ are Bourgain systems
	and $\mB$ is regular,
	$A$ is a subset of $B$ of relative density $\alpha$ and
	$X$ is a subset of $\dot{B}$ of relative density $\tau$.
	Assume also that $\dmB \leq_{\rho} \mB$ 
	with $\rho \leq c\kappa\alpha / d$ and write $f_A = 1_A - \alpha 1_B$.
	Then if
	\begin{align*}
		\sum_{\Spec_\eta(\mu_X)} |\wf_A|^2 \geq \kappa\alpha^2 b,
	\end{align*}
	there exists an $m$-controlled Bourgain system $\wmB$ such that
	\begin{align*}
		\cmB &= \dmB_{c/\dot{d^2} m} \wedge \wmB 
		\quad\text{is regular}, \\
		m &\leq C \eta^{-2} \ell(\tau), \\
		\| 1_A \ast \mu_{\cB} \|_\infty &\geq ( 1 + 2^{-3} \kappa ) \alpha . 
	\end{align*}
\end{proposition}

\begin{proof}
By Proposition~\ref{thm:local_specann}, 
$\Spec_\eta (\mu_X)$
is $\tfrac{1}{2}$-annihilated
by a regular Bourgain system of the form
$\cmB = \dmB_{c\dot{d}^2/m} \wedge \wmB$,
where $\wmB = \wmB'_{1/2}$ and
$\wmB'$ is an $m'$-controlled Bourgain system
with $m' \leq C \eta^{-2} \ell(\tau)$.
Note that by Lemma~\ref{thm:bsyst_dilation},
$\wmB$ is $O(m')$-controlled.
Applying then Lemma~\ref{thm:roth_L2tocorrel} with 
$\Delta = \Spec_\eta (\mu_X)$ and $T = \cmB \leq \dmB$
concludes the proof.
\end{proof}

We now take a big step forward and claim
that the following analog
of~\cite[Lemma~6.2]{sandersroth} holds.
This involves a careful examination of
the argument of~\cite{sandersroth}, 
and we regret imposing the double-checking
process below on the reader,
however past this point our argument is 
again self-contained.

\begin{proposition}[Iterative lemma on two scales]
\label{thm:roth_it2scales}
	Suppose that $\mB$, $\mB'$ are regular Bourgain systems,
	$A$ is a subset of $B$ of relative density $\alpha$
	and $A'$ is a subset of $B'$ of relative density $\alpha'$.
	Assume also that $\mB' \leq_\rho \mB$ with $\rho \leq c \alpha / d$.
	Then either
	\begin{enumerate}
		\item
		(Many three-term arithmetic progressions)
		\begin{align*}
				\langle 1_A \ast 1_{A'} , 1_{-A} \rangle_{L^2} \geq 
			\exp\big[ - C \alpha^{-1} \ell(\alpha') - C d' \ell(\alpha'/d') \big] bb' ,
		\end{align*}
		\item
		(Density increment)
		\newline
		there exists an $m$-controlled Bourgain system $\wmB$ with
		\begin{align*}
			\cmB &= \mB'_{(\alpha\alpha'/2d')^C} \wedge \wmB \quad\text{regular}, \\
			m &\leq C \alpha^{-1} \ell(\alpha)^3 \ell(\alpha'), \\
			\| 1_A \ast \mu_{\cB} \|_\infty &\geq ( 1 + 2^{-13} ) \alpha .
		\end{align*}
	\end{enumerate}
\end{proposition}

\begin{proof}
This is obtained by replacing each occurence of
the energy-increment lemma~\cite[Lemma~3.8]{sandersroth}
for Bohr sets by its Bourgain system counterpart, 
viz.~Proposition~\ref{thm:roth_L2incr}.
Essentially two types of $L^2$ density-increment
appear in Sanders' argument, 
and we now describe them,
using the notation of Proposition~\ref{thm:roth_L2incr}.
In every application of~\cite[Lemma~3.8]{sandersroth}
the Bourgain system $\dmB$
is (eventually) a dilate of the Bourgain system $\mB$
by a factor $(\alpha\alpha'/2d')^{O(1)}$,
and therefore we only need determine the parameters 
$\kappa,\eta,\tau$.

The first type of $L^2$ density-increment appears
in the proof of~\cite[Lemma~4.2]{sandersroth}
on p.~626 with parameters
$\kappa \asymp 1$,
$\eta \asymp \alpha^{1/2}$,
$\tau \gg \alpha'$,
so that $m \leq C\alpha^{-1} \ell(\alpha')$
upon applying Proposition~\ref{thm:roth_L2incr}.
The same density-increment is featured in
\cite[Proposition~4.1]{sandersroth} which is 
just an iteration of the previous lemma.

A second type of density-increment
arises in the proof of~\cite[Corollary~5.2]{sandersroth} on pp.~630--632
which involves certain densities $\sigma$ and $\lambda$,
and which features parameters 
$\kappa \asymp \lambda$,  $\eta \asymp 1$, 
\begin{align*}
	\tau \geq \exp[ - C \lambda^{-2} \ell(\sigma) \ell( \lambda\alpha )^2 \ell (\alpha) ]
	\quad\text{so that}\quad
	 m \leq C \lambda^{-2} \ell(\sigma) \ell( \lambda\alpha )^2 \ell (\alpha)
\end{align*}
upon applying Proposition~\ref{thm:roth_L2incr}.
This is finally combined with \cite[Proposition~4.1]{sandersroth}
on p.~633 to obtain \cite[Lemma~6.2]{sandersroth},
to the effect that we either have an $L^2$ density-increment
of the first type, or of the second type with
$\lambda \asymp 1$ and $\sigma \geq \exp[ - C \alpha^{-1} \ell(\alpha')]$,
and therefore such that $\kappa \asymp 1$
and $m \leq C \alpha^{-1} \ell(\alpha)^3 \ell(\alpha')$
in the application of Proposition~\ref{thm:roth_L2incr}.
Choosing $\mB'' = \mB'_{c\alpha'/d'}$ 
in (the Bourgain system version of)~\cite[Lemma~6.2]{sandersroth} 
and using Lemma~\ref{thm:bsyst_dilation}, we obtain
an alternative case (i) of the desired shape.

Since, by Lemma~\ref{thm:reg_averaging},
Bourgain systems satisfy the same regularity estimates as Bohr sets,
we may replace the latter by the former
and apply Proposition~\ref{thm:roth_L2incr} 
everywhere as claimed, thereby 
obtaining the desired iterative lemma.
Finally, the constant $2^{-13}$ may be extracted 
from~\cite{sandersroth} although its precise value is unimportant;
it is just convenient to write down an explicit value for
later computations.
\end{proof}

At this point we recall a simple technique,
originating in Bourgain's proof of Roth's 
theorem~\cite[(5.13)--(5.18)]{bourgainroth},
which allows one to pass from two scales to one
in iterative statements. 

\begin{lemma}
\label{thm:roth_modelling}
	Let $\theta \in (0,1]$ be a parameter.
	Suppose that $\mB$, $\mB'$, $\mB''$ are Bourgain systems,
	$\mB$ is regular and $A$ is a subset of $B$ 
	of relative density $\alpha$.
	Assume also that $\mB' \leq_\rho \mB$
	and $\mB'' \leq_\rho \mB$ with $\rho \leq c \theta \alpha / d$.
	Then either
	\begin{align*}
		\max\big( \| 1_A \ast \mu_{B'} \|_\infty ,  
		\| 1_A \ast \mu_{B''} \|_\infty \big)
		\geq \big( 1 + \tfrac{\theta}{2} \big) \alpha
	\end{align*} 
	or there exists $x$ such that
	$ 1_A \ast \mu_{B'}(x) \geq (1 - \theta) \alpha$
	and
	$ 1_A \ast \mu_{B''}(x) \geq (1 - \theta) \alpha$. 
\end{lemma}

\begin{proof}
A quick regularity computation via Lemma~\ref{thm:reg_averaging} yields
\begin{align*}
	\E_B ( 1_A \ast \mu_{B'} + 1_A \ast \mu_{B''} )
	&= \langle 1_A , \mu_B \ast \mu_{B'} \rangle
	+ \langle 1_A , \mu_B \ast \mu_{B''} \rangle \\
	&= 2\alpha + O( \rho d ) \\
	&\geq ( 2 - \tfrac{\theta}{2} )\alpha
\end{align*}
provided that $\rho \leq c \theta \alpha / d$.
By the pigeonhole principle, there exists $x \in G$ such that
\begin{align*}
	1_A \ast \mu_{B'}(x) + 1_A \ast \mu_{B''}(x) \geq ( 2 - \tfrac{\theta}{2} ) \alpha .
\end{align*}
Assuming that we are not in the first case of the lemma,
we have
\begin{align*}
	1_A \ast \mu_{B'}(x) 
	\geq ( 2 - \tfrac{\theta}{2} ) \alpha - ( 1 + \tfrac{\theta}{2} ) \alpha
	= ( 1 - \theta ) \alpha
\end{align*} 
and similarly for $1_A \ast \mu_{B''} (x)$.
\end{proof}

With this technique in hand, we may modify
Proposition~\ref{thm:roth_it2scales}
so as to make the iteration easier to perform.
Once this is done, Proposition~\ref{thm:roth_localroth}
is derived by a standard, yet computationally intensive 
iterative process.
For this argument to work however,
we need to make the assumption 
that the set $A$ contains no degenerate
arithmetic progressions at each step
of the iteration.

\begin{proposition}[Final iterative lemma]
\label{thm:roth_it1scale}
	Suppose that $G$ has odd order,
	$\mB$ is a regular Bourgain system, and
	$A$ is a subset of $B$ of relative density $\alpha$
	such that $A - A$ contains no element of order $2$.
	Then either
	\begin{enumerate}
		\item
		(Many three-term arithmetic progressions)
		\begin{align*}
			\langle 1_A \ast 1_A , 1_{2 \cdot A} \rangle_{L^2} \geq 
			\exp\big[ - C \alpha^{-1} \ell(\alpha) - C d \ell(\alpha/d) \big] \cdot b^2 ,
		\end{align*}
		\item
		(Density increment)
		\newline
		there exist Bourgain systems $\hmB$, $\wmB$ 
		and an element $u \in \{1,-2\}$ such that
		\begin{align*}
			&& \cmB &= \hmB \wedge \wmB \ \ \text{is regular}, && \\
			&& \hmB &= u \cdot \mB_{(\alpha/2d)^C}, &
			\wh{b} &\geq \exp\big[ - C d \ell(\alpha/d) \big] \cdot b, \\
			&& \wt{d} &\leq C \alpha^{-1} \ell(\alpha)^4,  &
			\wt{b} &\geq \exp[ - C \alpha^{-1} \ell(\alpha)^5 ], \\
			&& \| 1_A \ast \mu_{\cB} \|_\infty &\geq ( 1 + 2^{-16} ) \alpha. &&&&
		\end{align*}
	\end{enumerate}
\end{proposition}

\begin{proof}
Let $\theta = 2^{-15}$ and define
regular Bourgain systems $\mB' = \mB_{c\alpha/d}$ and $\mB'' = \mB'_{c'\alpha/d}$
with the help of Lemma~\ref{thm:reg_regularizing}.
Now apply Lemma~\ref{thm:roth_modelling} 
to $A$ and $\mB,\mB',\mB''$:
in the first case of that lemma, 
we are in the second case of the proposition,
while in the second case
we may find an element $x$ such that
$A' \coloneqq (A-x) \cap B' $ has relative
density $\alpha' \geq (1 - 2^{-15}) \alpha$ in $B'$,
and $A'' \coloneqq (A-x) \cap B''$
has relative density at least
$\tfrac{1}{2} \alpha$ in $B''$;
the latter weak bound suffices for our purposes.

We let $\hA'' = -2 \cdot A''$ and $\hmB'' = - 2 \cdot \mB''$, 
so that from the injectivity of $y \mapsto 2 \cdot y$ on $A''$
and the bound $|\hB''| \leq |B''|$,
we deduce that $\hA''$ has density 
at least $\tfrac{1}{2} \alpha$ in $\hB''$.
Furthermore, by Lemma~\ref{thm:bsyst_homimage}, 
we see that $\hmB''$ is a Bourgain system of
dimension at most $d''$ and, 
since $\hB''$ contains $\hA''$,
of density at least $\tfrac{1}{2} \alpha b''$.
Observe finally that 
with these choices of $A'$ and $\hA''$, we have
\begin{align}
\label{eq:roth_innerproduct}
	\langle 1_A \ast 1_A , 1_{2 \cdot A} \rangle_{L^2}
	= \langle 1_{A-x} \ast 1_{2x-2 \cdot A} , 1_{x-A} \rangle_{L^2}
	\geq \langle 1_{A'} \ast 1_{\smash{\hA''}} , 1_{-A'} \rangle_{L^2}.
\end{align}

We now apply Proposition~\ref{thm:roth_it2scales} to
the sets $A'$ and $\hA''$, 
located respectively in $B'$ and $\hB''$.
In the first case of that proposition,
it follows from~\eqref{eq:roth_innerproduct} and
Lemma~\ref{thm:bsyst_dilation} that 
we are in the first case of the proposition we seek to prove.
In the second case of Proposition~\ref{thm:roth_it2scales}, 
we obtain a regular Bourgain system 
$\cmB = \hmB \wedge \wmB$ where
\begin{align*}
	\hmB 
	= ( - 2 \cdot \mB'' )_{(\alpha/2d)^C} 
	= - 2 \cdot \mB''_{(\alpha/2d)^C}
	= - 2 \cdot \mB_{(\alpha/2d)^{C'}}
\end{align*}
and $\wmB$ is $C\alpha^{-1}\ell(\alpha)^4$-controlled, and such that
\begin{align*}
	\| 1_{A} \ast \mu_{\cB} \|_\infty 
	\geq \| 1_{A'} \ast \mu_{\cB} \|_\infty 
	\geq ( 1 + 2^{-13} ) \alpha' 
	\geq ( 1 + 2^{-14} ) \alpha .
\end{align*}
Applying Lemma~\ref{thm:bsyst_dilation}
to $\hmB = \hmB''_{(\alpha/2d)^C}$, recalling
that $\wh{b}'' \geq \tfrac{1}{2}\alpha b''$, 
and via Definition~\ref{thm:local_bsystcontrol}, 
we conclude that we are in the second case of the
proposition that we intend to prove.
\end{proof}

\textit{Proof of Proposition~\ref{thm:roth_localroth}.}
We construct iteratively sequences of subsets $A_i$
of regular Bourgain systems $\mB^{(i)}$ of density $\alpha_i$,
such that $A_i$ is contained in a translate of $A$.
Since $A_i - A_i$ is a subset $A - A$, it does not 
contain any element of order~$2$ either.
We initiate the iteration with 
$A_1 = A$ and $\mB^{(1)} = \mB$.

At each step we apply Proposition~\ref{thm:roth_it1scale}
to the set $A_i$, and in the first case of
that proposition we stop the iteration,
while in the second case we 
let $\mB^{(i+1)} = \cmB^{(i)}$
with the notation from there,
and we pick $x_i$ and
$A_{i+1} = ( A_i - x_i ) \cap \cB^{(i)}$
so that $A_{i+1}$ has relative density
$\alpha_{i+1} = \| 1_{A_i} \ast \mu_{ \smash{ \cB^{(i)} } } \|_\infty$ in $\cB^{(i)}$. 

Since $\alpha_{i+1} \geq (1+c) \alpha_i$ whenever
$A_{i+1}$ is defined, the iteration
proceeds for a number of steps bounded by $C \ell(\alpha)$.
At each step, we obtain Bourgain systems
$\hmB^{\,(i)}$ and $\wmB^{(i)}$ and an element 
$u_i \in \{1,-2\}$ such that 
\begin{align}
\label{eq:roth_newBi}
	\mB^{(i+1)} = \hmB^{\,(i)} \wedge \wmB^{(i)}
	\quad\text{is regular},
\end{align}
and, since $\alpha_i \geq \alpha$, such that
\begin{align}
	\label{eq:roth_dimdensityhat}
	\hmB^{\,(i)} &= u_i \cdot \mB^{(i)}_{(\alpha_i/2d_i)^C}, &
	\wh{b}_i &\geq \exp\big[ \! - C d_i \ell(\alpha/d_i) \big] \cdot b_i, \\
	\label{eq:roth_dimdensitytilde}
	\wt{d}_i &\leq C \alpha^{-1} \ell(\alpha)^4, &
	\wt{b_i} &\geq \exp\big[ \! - C \alpha^{-1} \ell(\alpha)^5 \big].
\end{align}

Iterating $i-1$ times \eqref{eq:roth_newBi} 
and \eqref{eq:roth_dimdensityhat},
we obtain a Bourgain system of the form
\begin{align*}
	\mB^{(i)} = \wmB^{(i-1)}
	\wedge u_{i-1} \cdot \Big( \dots
	u_2 \cdot ( \wmB_*^{(1)} 
	\wedge u_1 \cdot \wmB_* ) \dots \big) \Big)
\end{align*}
where the stars stand for certain dilations.
This is not exactly an intersection of Bourgain
systems, however the argument used in the proof
of Lemma~\ref{thm:bsyst_intersection} is easily
adapted to show that $\mB^{(i)}$ has dimension at most
\begin{align*}
	d_i \leq 2 ( d + \wt{d}_1 + \dotsb + \wt{d}_{i-1} ) .
\end{align*}
By \eqref{eq:roth_dimdensitytilde} and since $i \leq C\ell(\alpha)$,
this yields $d_i \leq 2d + C\alpha^{-1} \ell(\alpha)^5$.

Applying Lemma~\ref{thm:bsyst_intersection}
to the intersection \eqref{eq:roth_newBi},
and with \eqref{eq:roth_dimdensityhat}
and \eqref{eq:roth_dimdensitytilde}, 
we also obtain
\begin{align*}
	b_{i+1} 
	&\geq 4^{- ( \wh{d}_i + \wt{d_i} ) } \cdot 
	\wh{b}_i \cdot \wt{b_i} \\
	&\geq \exp\big[ \! - C ( \alpha^{-1} + d ) \ell(\alpha)^5 \ell( \alpha/d ) \big] \cdot b_i .
\end{align*}
Iterating this at most $C \ell(\alpha)$ times, we obtain
\begin{align*}
	b_i \geq \exp\big[ \! - C ( \alpha^{-1} + d ) \ell(\alpha)^6 \ell(\alpha/d) \big] \cdot b .
\end{align*}
When the algorithm stops, we have therefore
\begin{align*}
	\langle 1_{A_i} \ast 1_{A_i} , 1_{2 \cdot A_i} \rangle_{L^2}
	&\geq \exp\big[ \! - C \alpha^{-1} \ell(\alpha) - C d_i \ell(\alpha/d_i) \big] \cdot b_i^2.
\end{align*}
Inserting the bounds on $d_i$ and $b_i$ in the above,
and recalling that $A_i$ is contained in a translate of $A$,
this concludes the proof.

\section{From small doubling to three-term arithmetic progressions}
\label{sec:3aps}

This section is concerned with
the proof of Theorem~\ref{thm:me_2xto3aps}
and the related Corollary~\ref{thm:me_restrictedss}.
As mentioned before, an extremely
important tool for us is the recent correlation-based
Bogolyubov-Ruzsa lemma of Sanders~\cite{sandersBR}.
In our situation, it serves to pass 
from a set of small doubling to one with
high density in a coset progression,
which is a particular type of Bourgain system. 
The local Sanders-Roth theorem of the
previous section may then be applied
to this new set, to show that it contains a 
nontrivial three-term arithmetic progression;
this is the main observation of this paper.
We now quote the main result of~\cite{sandersBR},
with a minor tweak to ensure regularity.

\begin{proposition}[Correlation Bogolyubov-Ruzsa lemma \cite{sandersBR}]
\label{thm:br_correl}
	Let $K \geq 1$ be a parameter, 
	and suppose that $A$ is a subset of $G$
	such that $|A+A| \leq K |A|$.
	Then there exists a $d$-dimensional 
	coset progression $M$ inducing 
	a regular Bourgain system and such that 
	\begin{align*}
		\| 1_A \ast \mu_M \|_{\infty} &\geq \tfrac{1}{2K}, \\
		d &\leq C (\log K)^6, \\
		|M| &\geq \exp\big[ - C(\log K)^6 (\log\log K) \big] \cdot |A| .
	\end{align*}
\end{proposition}

\begin{proof}
Without the regularity condition, this
is~\cite[Theorem~10.1]{sandersBR} 
with $A = S$ and $\eps = \tfrac{1}{2}$.
To obtain regularity,
one may simply follow the proof in~\cite{sandersBR},
stopping just before the application
of~\cite[Lemma~10.2]{sandersBR},
and dilating by a certain constant factor 
the coset progression $M$ obtained
at this point.
By Lemmas~\ref{thm:bsyst_dilation}
and~\ref{thm:reg_regularizing},
one may choose this constant so 
that the dilated induced Bourgain
system is regular, while losing at most
a factor $e^{-C(\log K)^6}$ in size,
and the rest of the proof goes unchanged.
\end{proof}

It is crucial for our argument that this
statement makes no assumption of density on 
the set $A$, whereas the earlier Bogolyubov-Chang-type 
lemma~\cite[Proposition~6.1]{sanders3aps} 
used by Sanders does.
In terms of bounds, we could also allow for 
$d \leq K^{1+o(1)}$ and 
${|M| \geq e^{ - C K^{1+o(1)} } |A|}$ 
in Proposition~\ref{thm:br_correl},
without affecting the quality of 
bounds in Theorem~\ref{thm:me_2xto3aps};
however we do not know of any
argument significantly simpler than that 
of~\cite{sandersBR} to obtain such estimates.

We now present the proof of Theorem~\ref{thm:me_2xto3aps},
following the usual approach of estimating
the total number of three-term arithmetic progressions, 
only to compare it later to the number of trivial ones.
Corollary~\ref{thm:me_restrictedss} then
follows by inserting the bound of Theorem~\ref{thm:me_2xto3aps}
into the argument of~\cite{sanders3aps}.

\begin{proposition}
\label{thm:3aps_many3aps}
	Let $K \geq 1$ be a parameter.
	Suppose that $A$ is a subset of $G$ 
	such that $|A+A| \leq K|A|$ and 
	$A-A$ contains no element of order $2$.
	Then
	\begin{align*}
		\langle 1_A \ast 1_A , 1_{2 \cdot A} \rangle_{L^2}
		\geq \exp\big[ \! - C K (\log K)^7 \big] 
		\cdot \mu_G(A)^2 .
	\end{align*}
\end{proposition}

\begin{proof}
Let $M$ be the coset progression given by
Proposition~\ref{thm:br_correl},
and write $\mathcal{M}$ for its induced 
regular Bourgain system.
By the correlation conclusion,
we may pick an element $x$ such that  
$A' = ( A - x ) \cap M$ has relative
density $\tfrac{1}{2K}$ in $M$.
Applying then Proposition~\ref{thm:roth_localroth}
to $A'$ and $\mM$, we obtain
\begin{align*}
	\langle 1_A \ast 1_A , 1_{2 \cdot A} \rangle_{L^2} 
	\geq 
	\langle 1_{A'} \ast 1_{A'} , 1_{2 \cdot A'} \rangle_{L^2}
	\geq
	\exp[ - C ( K + d ) 
	(\log K)^6 (\log Kd) ] 
	\cdot \mu_G(M)^2.
\end{align*}
This yields the desired estimate
upon inserting the bounds from 
Proposition~\ref{thm:br_correl}.
\end{proof}

\textit{Proof of Theorem~\ref{thm:me_2xto3aps}.}
Write $K = |A+A| / |A|$.
If $A - A$ contains an element $x-y$ of order $2$, 
we readily find a nontrivial,
degenerate arithmetic progression $(x,y,x)$ in $A$.
Otherwise, Proposition~\ref{thm:3aps_many3aps} tells us
that $A$ possesses at least
$e^{- C K (\log K)^7 } |A|^2$
three-term arithmetic progressions, while the number
of trivial ones is at most $|A|$.
By the assumption on $K$,
we are then ensured to find at least 
one nontrivial arithmetic progression in $A$.
\qed

\textit{Proof of Corollary~\ref{thm:me_restrictedss}.}
It suffices to insert the bounds of Theorem~\ref{thm:me_2xto3aps}
in the proof of~\cite[Theorem~1.5]{sanders3aps} on pp.~230--231.
\qed

\section{From small doubling to long arithmetic progressions}
\label{sec:long}

In this section we derive Theorem~\ref{thm:me_2xtolongaps},
basing ourselves on the approach of Croot et al.~\cite{CLS},
which divides roughly into three steps. 
In the first step, one produces a large, structured
set of almost periods of the convolution of the set 
$A$ under consideration with itself.
The second step is to show, by a packing argument,
that the set $A + A$ necessarily contains a translated 
copy of subset of this set of almost-periods
of a certain size.
The third step is to pick such a subset 
with basic additive structure,
such as an arithmetic progression.

The original argument of~\cite{CLS}
is based on Ruzsa's modelling lemma~\cite{ruzsamodelling},
which has no efficient equivalent for 
general abelian groups,
and therefore we need to use again 
a modelling approach based on 
the Bogolyubov-Ruzsa lemma of Sanders.
In contrast with the previous section however,
we now need a version of this lemma
that provides us with a containment 
conclusion, and for this 
we quote~\cite[Theorem~1.1]{sandersBR}.

\begin{proposition}[Containment Bogolyubov-Ruzsa lemma~\cite{sandersBR}]
\label{thm:br_cont}
	Let $K \geq 1$ be a parameter, 
	and suppose that $A$ is a subset of $G$
	such that $|A + A| \leq K |A|$.
	Then there exists a $d$-dimensional 
	coset progression $M$ contained
	in $2A - 2A$ and such that
	\begin{align*}
		d \leq C (\log K)^6  
		\quad\text{and}\quad
		|M| \geq \exp\big[ - C(\log K)^6 (\log\log K) \big] \cdot |A|.
	\end{align*}
\end{proposition}

As noted in~\cite[Section~3]{sandersBR}, this version 
can be deduced from Proposition~\ref{thm:br_correl}.
The containment conclusion is sufficient in our situation, 
because the Croot-Sisask lemma works
under a doubling hypothesis, whereas the
iterative argument used in the proof of Roth's theorem
requires an assumption of density instead.
Our reason for emphasizing this point is 
that the containment version above 
is easier to obtain than the correlation one,
and is explained in depth in a survey 
by Sanders~\cite{sandersPFR}.
Although the type of structure obtained there is
different, consisting of a convex coset progression instead, 
this would not affect our argument much since
this object is also a Bourgain system, 
as can be seen from~\cite[Section~4]{sandersPFR}.

We now proceed to the proof, starting with 
the following lemma which serves to collect together certain
computations from~\cite{CLS} on 
$L^p$ and $L^{p/2}$ norms of convolutions.

\begin{lemma}
\label{thm:long_lpnorms}
	Let $p \geq 2$ and $K \geq 1$ be parameters.
	Suppose that $A$ is a subset of $G$
	such that ${|A+A| \leq K|A|}$.
	Then
	\begin{align*}
		\mu_G ( A + A )^{1/p} 
		\leq K^{1/2} \| 1_A \ast \mu_A \|_{p/2}^{1/2}
		\quad\text{and}\quad
		\| 1_A \ast \mu_A \|_{p/2}^{1/2} 
		\leq K^{1/2} \| 1_A \ast \mu_A \|_p .
	\end{align*}	
\end{lemma}

\begin{proof}
By Hölder's inequality we have
\begin{align*}
	\mu_G (A) 
	= \E_G 1_A \ast \mu_A 
	\leq \mu_G ( A + A )^{1-2/p} \| 1_A \ast \mu_A \|_{p/2},
\end{align*}
from which the first estimate follows upon rearranging
and taking square roots.
To obtain the second, apply Cauchy-Schwarz
and the first estimate in
\begin{align*}
	\big[ \E_G \, ( 1_A \ast \mu_A )^{p/2} \big]^2
	\leq \mu_G ( A + A ) \| 1_A \ast \mu_A \|_p^p 
	\leq K^{p/2} \| 1_A \ast \mu_A \|_{p/2}^{p/2} \| 1_A \ast \mu_A \|_p^p .
\end{align*}
The result follows upon taking {$p$-th} roots,
then dividing both sides
by ${\| 1_A \ast \mu_A \|_{p/2}^{1/2}}$.
\end{proof}

An important tool from~\cite{CLS} 
is a version of the Croot-Sisask lemma~\cite{CS}
that serves to smooth the convolution of two sets
by an iterated convolution factor.
The precise statement we need is a standard
consequence of~\cite[Theorem~6.1]{CLS};
an exposition of it by the author 
may be found in~\cite[Section~7]{meexpocs}.

\begin{lemma}[Croot-Sisask $L^p$-smoothing]
\label{thm:long_cs}
	Let $K, L \geq 1$, ${ \theta \in (0,K^{-1/2}] }$,
	$p \in 2 \N$, $\ell \in \N$ be parameters.
	Suppose that $A,S,T$ are subsets of $G$
	such that $|A+S| \leq K |A|$ and $|S+T| \leq L |S|$. 
	Then there exists a subset $X$ of $T$ of size 
	$|X| \geq (2L)^{- C p \ell^2 / \theta^2 } |T|$ such that
	\begin{align*}
		\| 1_A \ast \mu_S - 1_A \ast \mu_S \ast \lambda_X^{(\ell)} \|_p
		\leq \theta \| 1_A \ast \mu_S \|_{p/2}^{1/2}
	\end{align*}
	where $\lambda_X = \mu_X \ast \mu_{-X}$.
\end{lemma}

As anticipated, our first step is to
produce a set of almost-periods of the
convolution of a small doubling set with itself.
Following~\cite{CLS}, this is done by first
smoothing this convolution by the iterated convolution
of a certain set $X$,
with the difference that this set 
is now localized to a Bourgain system,
which is taken to be a coset progression later on.
Via the Fourier transform,
any set annihilating the large spectrum of $X$
induces a set of almost-periods of the smoothed convolution,
and via the results of Section~\ref{sec:local},
we may choose this annihilator to be a 
large Bourgain system.
Here we make a small parenthesis on notation:
throughout this section,
$a \sim b$ stands for $b/2 \leq a \leq 2b$.

\begin{proposition}
\label{thm:long_almostp}
	Let $K \geq 1$ and $p \in 2 \N$ be parameters.
	Suppose that $\mB$ is a regular Bourgain system
	and $A$ is a subset of $G$
	such that $|A + A| \leq K|A|$
	and $B \subset 2A - 2A$.
	Then there exist $m \geq 1$ and Bourgain systems $\cmB,\wmB$
	such that $\wmB$ is $m$-controlled and
	\begin{align*}
		\cmB &= \mB_{ c / ( K d^2 m ) } \wedge \wmB_{c/K}, \\
		m &\leq C p K ( \log K )^3, 
	\end{align*}
	and for every $x \in \cB$,
	\begin{align*}
		\| 1_A \ast \mu_A - \tau_x 1_A \ast \mu_A \|_p 
		\leq \tfrac{1}{2} \| 1_A \ast \mu_A \|_p .
	\end{align*}
\end{proposition}

\begin{proof}
First observe that, by the Plünnecke-Ruzsa-Petridis
inequality~\cite{petridis},
\begin{align*}
	| A + B | \leq | 3A - 2A | \leq K^5 |A|,
\end{align*}
and therefore we may apply Lemma~\ref{thm:long_cs} with 
$(S,T) = (A,B)$ and $L = K^5$, 
for parameters $\theta$ and $\ell$ to be determined later.
This yields a subset $X$ of $B$ of relative density $\tau$ such that
\begin{gather}
	\label{eq:long_Xdensity}
	\tau \geq \exp\big[ -C p \ell^2 \theta^{-2} \log K \big], \\
	\label{eq:long_smoothing}
	\| 1_A \ast \mu_A - 1_A \ast \mu_A \ast \lambda_X^{(\ell)} \|_p
	\leq \theta \| 1_A \ast \mu_A \|_{p/2}^{1/2}.
\end{gather}

We write $I$ for the identity operator on functions,
and given $x \in G$ we define the function 
$\wh{x} : \dual \rightarrow G$ which maps 
$\gamma$ to $\gamma(x)$.
Consider now an arbitrary element $x$ of $G$,
then by the triangle inequality and \eqref{eq:long_smoothing},
we have
\begin{align*}
	\| ( I - \tau_x ) 1_A \ast \mu_A \|_p
	&\leq\  \| ( I - \tau_x ) ( 1_A \ast \mu_A - 1_A \ast \mu_A \ast \lambda_X^{(\ell)} ) \|_p \\
	& \mspace{80.0mu} + \| 1_{(A + A) \cup (A + A-x)} \cdot 
	(I - \tau_x ) 1_A \ast \mu_A \ast \lambda_X^{(\ell)} \|_p \\
	&\leq\ 2\theta \| 1_A \ast \mu_A \|_{p/2}^{1/2}
	+ 2\mu_G(A + A)^{1/p} \| (I - \tau_x) 1_A \ast \mu_A \ast \lambda_X^{(\ell)} \|_\infty .
\end{align*}
By Parseval, we have further
\begin{align}
\label{eq:long_almostpparseval}
	\| ( I - \tau_x ) 1_A \ast \mu_A \|_p
	\leq 2\theta \| 1_A \ast \mu_A \|_{p/2}^{1/2}
	+ 2 \mu_G(A + A)^{1/p} \ssum_{\dual} |\wone_A| |\wmu_A| |\wmu_X|^{2\ell} |1 - \wh{x}|.
\end{align}

Invoking now Proposition~\ref{thm:local_specann} 
with a parameter $\nu \in (0,1]$, 
and recalling \eqref{eq:long_Xdensity},
we infer that $\Spec_{1/2}(\mu_X)$ is $\nu$-annihilated
by $\cmB = \mB_{c\nu/d^2 m} \wedge \wmB_\nu$,
where $\wmB$ is an $m$-controlled Bourgain system with 
$m \leq C p \ell^2 \theta^{-2} \log K$.
From now on we restrict to $x \in \cB$, so that, 
by considering separately the summation over 
$\Spec_{1/2}(\mu_X)$ in \eqref{eq:long_almostpparseval},
we obtain
\begin{align*}
	\| ( I - \tau_x ) 1_A \ast \mu_A \|_p
	&\leq  2\theta \,\| 1_A \ast \mu_A \|_{p/2}^{1/2}
	+ 2(\nu + 2^{1 - 2\ell}) \, \mu_G(A + A)^{1/p} \ssum_{\dual} |\wone_A| |\wmu_A| .
\end{align*}
By Parseval we know that
$\sum_{\dual} |\wone_A| |\wmu_A| = 1$.
Applying finally Lemma~\ref{thm:long_lpnorms}, we obtain 
\begin{align*}
	\| ( I - \tau_x ) 1_A \ast \mu_A \|_p
	&\leq \big( 2\theta + 2\nu K^{1/2} + 2^{2-2\ell} K^{1/2} \big) 
	\, \| 1_A \ast \mu_A \|_{p/2}^{1/2} \\
	&\leq \big( 2\theta + 2\nu K^{1/2} + 2^{2-2\ell} K^{1/2} \big) 
	K^{1/2} \| 1_A \ast \mu_A \|_p.
\end{align*}
Choosing $\theta = K^{-1/2}/8$, 
$\nu = K^{-1} / 16$
and $\ell \sim C\log K$, 
we obtain the desired $L^p$-estimate,
and the bound on $m$ follows by inserting
the value of these parameters.
\end{proof}

Secondly, we need the following packing
argument which may be extracted from 
the computations of~\cite{CLS},
but whose proof we include for completeness.
In practice we specialize $f$ below to
$1_A \ast \mu_A$ which has $A+A$ as support.

\begin{lemma}
\label{thm:long_packing}
	Let $p \geq 2$ be a parameter.
	Suppose that $f : G \rightarrow \C$ and $R \subset G$
	are such that, for all $t \in R$,
	\begin{align*}
		\| (I - \tau_t) f \|_p 
		\leq \tfrac{1}{2} \| f \|_p .
	\end{align*}
	Then for every subset $T$ of $R$ of size $|T| < 2^p$,
	there exists a translate $x \in G$ such that $x + T \subset \Supp(f)$.
\end{lemma}

\begin{proof}
Given a subset $T$ of $R$, consider the quantity 
\begin{align*}
	I \coloneqq \sum_{t \in T} 
	\| f - \tau_t f \|_p^p,
\end{align*}
so that by the assumptions of the lemma, we have at once 
$I \leq |T| \cdot 2^{-p} \| f \|_p^p$.

Now assume for contradiction that for every $x \in G$,
the translate $x + T$ is not contained in $\Supp(f)$;
then for every $x \in G$ we may find an element $t \in T$
such that $f ( x + t ) = 0$.
Exchanging summations, this yields the lower bound
\begin{align*}
	I = \E_G \ssum_{t \in T} | f - \tau_t f |^p
	\geq \E_G |f|^p.
\end{align*}
Combining both bounds on $I$, we obtain
\begin{align*}
	\| f \|_p^p 
	\leq |T| 2^{-p} \| f \|_p^p .
\end{align*}
We obtain a contradiction if $|T| < 2^p$,
and therefore we find a translated copy 
of $Y$ in the support of $f$ in that case.
\end{proof}

Last, we need an analog for Bourgain systems
in abelian groups of the well-known fact,
used in~\cite{CLS}, that 
Bohr sets of $\Z_N$ of radius $\delta$ and 
dimension $d$ contain arithmetic 
progressions of length $\delta N^d$.

\begin{lemma}
\label{thm:long_apsginbsyst}
	Suppose that $\mB$ is a Bourgain system 
	of dimension $d$ and $h \geq d$,
	and assume that $|B| \geq 2^{6h}$.
	Then there exists a subset $T$ of $B$, 
	which is either a proper arithmetic progression
	or a subgroup, of size
	$\tfrac{1}{4} |B|^{1/4h} \leq |T| \leq |B|^{1/2h}$.
\end{lemma}

\begin{proof}
Let $ \eta = 2 |B|^{-1/2h} \in (0,2^{-2}]$ 
so that, by Lemma~\ref{thm:bsyst_dilation}, we have
\begin{align*}
	|B_\eta| 
	\geq \exp\big[ \log |B| - d \log \tfrac{2}{\eta} \big] 
	\geq |B|^{1/2}.
\end{align*}
Let $N = \lfloor \eta^{-1/2} \rfloor$, so that 
we have a sumset containment
\begin{align}
\label{eq:long_Bcontainment}
	N^2 B_\eta \subset B_{N^2 \eta} \subset B .
\end{align}
Since $\eta^{-1/2} \geq 2$, we have also
$ \tfrac{1}{2} \eta^{-1/2} \leq N \leq \eta^{-1/2}$.

We are now in one of two cases.
In the first, there exists an element 
$x$ in $B_\eta$ of order $N$,
thus the arithmetic progression 
$T = [0,N-1]_\Z \cdot x$ is proper
and, by \eqref{eq:long_Bcontainment}, contained in $B$.
Since $|T| = N$, we have also
${\tfrac{1}{4} |B|^{1/4h} \leq |T| \leq |B|^{1/4h}}$.

In the second case, every element of $B_\eta$ has order at most $N$.
Since $|B_\eta| \geq |B|^{1/2} \geq N$,
we may pick $N-1$ distinct nonzero elements 
$x_1,\dots,x_{N-1} \in B_\eta$ and consider 
the subgroup $T$ they generate, viz.
\begin{align*}
	T 
	= \langle x_1, \dots , x_{N-1} \rangle_\Z
	= [0,N-1]_\Z \cdot x_1 + \dotsb + [0,N-1]_\Z \cdot x_{N-1}.
\end{align*}
By \eqref{eq:long_Bcontainment} it follows again 
that $T$ is contained in $B$,
and the size of $T$ satisfies
\begin{align*}
	\tfrac{1}{4} |B|^{1/4h} \leq N 
	\leq |T| \leq N^2 \leq |B|^{1/2h}.
\end{align*} 
\end{proof}

We are now ready to combine the previous propositions
into a proof of Theorem~\ref{thm:me_2xtolongaps}.
\smallskip

\textit{Proof of Theorem~\ref{thm:me_2xtolongaps}.}
By Proposition~\ref{thm:br_cont}, 
we may find a $d$-dimensional coset 
progression $M \subset 2A - 2A$ such that
\begin{align}
	\label{eq:long_cprogbounds}
	d \leq (\log K)^{O(1)}
	\quad\text{and}\quad
	|M| \geq \exp\big[ - (\log K)^{O(1)} \big] \cdot |A|.
\end{align}
Up to dilating $M$ by a constant factor,
which preserves the above bounds by 
Lemma~\ref{thm:bsyst_dilation},
we may assume via Lemma~\ref{thm:reg_regularizing}
that $M$ induces a regular Bourgain system $\mM$.
By Lemma~\ref{thm:cprog_systdim},
that system also satisfies the dimension 
bound~\eqref{eq:long_cprogbounds}.

Applying now Proposition~\ref{thm:long_almostp} 
with $\mB = \mM$ and a parameter
$p \in 2\N$ to be determined later,
we obtain Bourgain systems $\cmB,\wmB$ such that
\begin{align}
	\label{eq:long_intersec}
	\cmB &= \mM_{ (1/2dpK)^{O(1)} } \wedge \wmB_{c/K}, \\
	\label{eq:long_intersecdim}
	\wt{d} &\leq C p K (\log K)^3, \\
	\label{eq:long_intersecdens}
	\wt{b} &\geq \exp\big[ - C p K (\log pK ) ( \log K )^3  \big],
\end{align}
where we have unfolded Definition~\ref{thm:local_bsystcontrol}, 
and such that
\begin{align}
\label{eq:long_packingineq}
	\| (I - \tau_x) 1_A \ast \mu_A \| 
	\leq \tfrac{1}{2} \| 1_A \ast \mu_A \|_p
	\quad\text{for all}\quad
	x \in \cB.
\end{align}
Applying Lemma~\ref{thm:bsyst_intersection} 
to the intersection~\eqref{eq:long_intersec}, 
and considering~\eqref{eq:long_cprogbounds}
and \eqref{eq:long_intersecdim},
we obtain
\begin{align*}
	\cj{d} \ll (\log K)^{O(1)} + p K (\log K)^3 \ll p K ( \log K )^3
\end{align*}
and we let $h = C p K (\log K)^3 \geq \cj{d}$.
By Lemmas~\ref{thm:bsyst_dilation}
and~\ref{thm:bsyst_intersection},
we also obtain
\begin{align*}
	\mu_G(\cB) \geq \exp\big[ - C d (\log dpK) \big] \mu_G(M) 
	\cdot \exp\big[ - C\wt{d} \log K \big] \wt{b}
\end{align*}
and therefore, by~\eqref{eq:long_cprogbounds},
\eqref{eq:long_intersecdim}
and~\eqref{eq:long_intersecdens},
\begin{align*}
	|\cB| 
	\geq \exp\big[ - C p K ( \log pK ) (\log K)^3 \big] \cdot |A|.
\end{align*}

Both the conditions $|\cB| \geq |A|^{1/2}$ 
and $|\cB| \geq 2^{6h}$ 
are satisfied provided
\begin{align}
	\label{eq:long_hardcondition}
	p K (\log pK) (\log K)^3 \leq c \log |A|. 
\end{align}
Considering that $\cB \subset M \subset 2A - 2A$, 
we thus have a rough estimate 
$|A|^{1/2} \leq |\cB| \leq |A|^4$.
By Lemma~\ref{thm:long_apsginbsyst},
we may therefore find a subset $T$ of $\cB$, 
which is either a proper arithmetic progression
or a subgroup, of size bounded by
\begin{align*}
	\tfrac{1}{4} |A|^{1/8h} 
	\leq \tfrac{1}{4} |\cB|^{1/4h}
	\leq |T| 
	\leq |\cB|^{1/2h}
	\leq |A|^{2/h}.
\end{align*}
Recalling our choice $h = C p K (\log K)^3$
and \eqref{eq:long_hardcondition}, 
this shows that
\begin{align*}
	|T| = \exp\bigg[ \Theta\Big( \frac{\log |A|}{pK (\log K)^3} \Big)\bigg].
\end{align*}
The condition $|T| < 2^p$ is therefore 
satisfied if we choose
\begin{align*}
	p \sim C \bigg( \frac{\log |A|}{K (\log K)^3} \bigg)^{1/2}.
\end{align*} 
It remains to check the conditions $p \geq 2$
and~\eqref{eq:long_hardcondition}; 
these are seen to be satisfied for
\begin{align*}
	K \leq \frac{c\log |A|}{(\log\log |A|)^5}
\end{align*}
after a tedious, yet elementary computation.
This yields the final size estimate
\begin{align*}
	|T| = \exp\bigg[ \Theta\Big( \frac{\log |A|}{K (\log K)^3} \Big)^{1/2} \bigg]
\end{align*}
and since we verified the conditions $|T| < 2^p$
and~\eqref{eq:long_packingineq}, 
an application of Lemma~\ref{thm:long_packing}
with $f = 1_A \ast \mu_A$ and $R = \cB$ concludes the proof.
\qed

\section{Remarks}
\label{sec:remarks}

In this section we collect together 
certain remarks of expository 
or exploratory nature which have
not found their way into the main text.

We first wish to explain in more detail how 
Theorem~\ref{thm:san_2xto3apsZ} follows
from the results of the literature.
Consider a set of integers $A$ of doubling $K$,
then for the purpose of finding arithmetic progressions in $A$,
we may instead assume that $A$ is a 
subset of a cyclic group of odd order
of density $\gg K^{-4}$ and doubling $K$,
via a partial Freiman isomorphy~\cite{ruzsamodelling}.
Applying~\cite[Proposition~6.1]{sanders3aps} to $A$,
one obtains a regular Bohr set of dimension
$d \ll K \log K$ and density $b \geq \exp[ - C K (\log K)^2]$,
on which a certain translate of $A$ has density $\gg K^{-1}$.
In that setting, Proposition~\ref{thm:roth_localroth}
of this article is just~\cite[Theorem~1.1]{sandersroth},
initializing the iterative argument from there 
on a Bohr set instead of the whole group;
there is no need to consider Bourgain systems
or $2$-torsion.
Proposition~\ref{thm:roth_localroth} thus specialized
shows that $A$ contains at least 
$\exp[ -C K (\log K)^8] \cdot |A|^2$
three-term arithmetic progressions,
and therefore at least one nontrivial
progression for $K=|A + A|/|A|$ 
in the range specified by 
Theorem~\ref{thm:san_2xto3apsZ}.


Secondly, we remark that the modelling argument
used in Sections~\ref{sec:3aps} and~\ref{sec:long}
could likely be adapted to other problems on dense sets,
such as solving translation-invariant equations
or finding long arithmetic progressions
in $A + A + A$, to obtain
a generalization of these results
to the case of sets of small doubling 
in an arbitrary abelian group. 
However, it is not clear to the author
whether it is worth pursuing such generalizations,
given the current lack of combinatorial applications 
of the kind of Corollary~\ref{thm:me_restrictedss}
for results of this type.

\bibliographystyle{amsplain}
\bibliography{smalldoubling_arxiv2}

\bigskip

\textsc{\footnotesize Département de mathématiques et de statistique,
Université de Montréal,
CP 6128 succ. Centre-Ville,
Montréal QC H3C 3J7, Canada
}

\textit{\small Email address: }\texttt{\small henriot@dms.umontreal.ca}

\end{document}